\documentclass[10pt]{article}
\usepackage{a4wide,graphicx,float,latexsym,amssymb,subeqnarray,amsmath,amsfonts}

\def\R{\mathbb{R}}

\newtheorem{theorem}{Theorem}

\newtheorem{lemma}[theorem]{Lemma}

\newenvironment{proof}{\noindent {\it Proof}~}{}

\title{Stability of ADI schemes for multidimensional\\ diffusion
equations with mixed derivative terms}

\author{K.~J.~in 't Hout\footnote{Department of Mathematics and Computer Science,
University of Antwerp, Middelheimlaan 1, B-2020 Antwerp, Belgium.
\mbox{Email}: \texttt{\{karel.inthout,chittaranjan.mishra\}@ua.ac.be}.}
~and C. Mishra\footnotemark[\value{footnote}]
}
\date{April 30, 2012}

\begin{document}

\maketitle

\begin{abstract}
\noindent
In this paper the unconditional stability of four well-known
ADI schemes is analyzed in the application to time-dependent
multidimensional diffusion equations with mixed derivative
terms.
Necessary and sufficient conditions on the parameter $\theta$
of each scheme are obtained that take into account the actual
size of the mixed derivative coefficients.
Our results generalize results obtained previously by Craig
\& Sneyd (1988) and In 't Hout \& Welfert (2009).
Numerical experiments are presented illustrating our main
theorems.
\end{abstract}

\section{Introduction}\label{intro}
\subsection{Multidimensional diffusion equations with mixed derivative terms}
This paper is concerned with stability in the numerical solution of
initial-boundary value problems for time-dependent multidimensional
diffusion equations containing mixed spatial-derivative terms,
\begin{equation}
\label{PDE}
\frac{\partial u}{\partial t} = \sum_{i\ne j} d_{ij} u_{x_ix_j}
+ d_{11} u_{x_1x_1}+ d_{22} u_{x_2x_2} + \cdots + d_{kk} u_{x_kx_k}
\end{equation}
Here $k\ge 2$ is any given integer and $D=(d_{ij})_{1\le i,j\le k}$\,
denotes any given real, symmetric, positive semidefinite $k\times k$
matrix.
The spatial domain is taken as $\Omega=(0,1)^k$\, and $t>0$.

Multidimensional diffusion equations with mixed derivative terms play an
important role, notably, in the field of financial option pricing theory.
Here these terms represent correlations between underlying stochastic
processes, such as for asset prices, volatilities or interest rates,
see e.g. \cite{AP10,S04,TR00,W99}.

In this paper we consider for given $\gamma \in [0,1]$ the following
useful condition on the relative size of the mixed derivative
coefficients,
\begin{equation}\label{gamma}
|d_{ij}|\le \gamma\, \sqrt{d_{ii}d_{jj}} \quad
\text{whenever}~~1\le i,j\le k,~i\neq j.
\end{equation}
Clearly, $\gamma=0$ yields the smallest possible mixed derivative
coefficients; they all vanish in this case.
Subsequently, $\gamma=1$ admits the largest possible mixed derivative
coefficients, given that matrix $D$ is positive semidefinite.
In most applications in financial mathematics it turns out that
(\ref{gamma}) holds with certain known $0< \gamma < 1$.
More precisely, such $\gamma$ can be determined from the pertinent
correlation coefficients.
The aim of our present paper is to effectively use this knowledge
about $\gamma$ in the stability analysis of numerical schemes,
discussed below, so as to arrive at weaker sufficient stability
conditions than those obtained when considering $\gamma=1$.

\subsection{Alternating Direction Implicit schemes}\label{ADI}
Semidiscretization of initial-boundary value problems for
multidimensional diffusion equations (\ref{PDE}) leads to initial
value problems for very large systems of stiff ordinary differential
equations,
\begin{equation}\label{ODE}
U'(t) = A U(t) + g(t)\quad (t\ge 0), \quad U(0)=U_0\,.
\end{equation}
Here $A$ is a given real $m\times m$ matrix and $g(t)$ (for $t\ge 0$)
and $U_0$ are given real $m\times 1$ vectors.
For the efficient numerical solution of (\ref{ODE}), splitting
schemes of the Alternating Direction Implicit (ADI) type form an
attractive means.
These schemes employ a splitting of the matrix $A$ that corresponds
to the different spatial directions, leading to a substantial
reduction in the amount of computational work per time step compared
to common implicit methods such as Crank--Nicolson.
ADI schemes constitute a popular class of numerical methods for
multidimensional equations in financial mathematics, where mixed
derivative terms are ubiquitous.

Originally ADI schemes were proposed and analyzed for equations
without mixed derivatives, see e.g. \cite{Br61,D62,DR56}.
To formulate ADI schemes in the presence of mixed derivative terms,
consider a splitting of the matrix $A$ into
\begin{equation}\label{splitting}
A = A_0 + A_1 + \cdots + A_k\,,
\end{equation}
where $A_0$ represents all mixed derivative terms in (\ref{PDE}) and
$A_j$ represents the derivative term in the $j$-th spatial direction
for $j=1,2,\ldots,k$.
Split the vector function $g$ in an analogous way,
\[
g=g_0+g_1+\cdots +g_k,
\]
and define
\[F(t,\xi )=A\xi +g(t) ~~{\rm and}~~ F_j(t,\xi )=A_j\xi +g_j(t) \quad {\rm whenever}~
t\ge 0,~ \xi \in\R^m,~ 0\le j\le k.
\]
Let parameter $\theta >0$ be given. Let time step $\Delta t >0$ and
temporal grid points $t_n=n \Delta t$ with integer $n\ge 0$.
For the numerical solution of the semidiscrete system (\ref{ODE}), we
study in this paper four ADI schemes, each successively generating for
$n=1,2,3,\ldots$ approximations $U_n$ to $U(t_n)$:
\vskip0.5cm
\noindent
\textit{Douglas (Do) scheme}
\begin{equation}\label{Do}
\left\{\begin{array}{l}
Y_0 = U_{n-1}+\Delta t F(t_{n-1},U_{n-1}), \\\\
Y_j = Y_{j-1}+\theta\Delta t \left(F_j(t_n,Y_j)-F_j(t_{n-1},U_{n-1})\right),
\quad j=1,2,\ldots,k, \\\\
U_n = Y_k
\end{array}\right.
\end{equation}
\\\\
\noindent
\textit{Craig--Sneyd (CS) scheme}
\begin{equation}\label{CS}
\left\{\begin{array}{l}
Y_0 = U_{n-1}+\Delta t F(t_{n-1},U_{n-1}), \\\\
Y_j = Y_{j-1}+\theta\Delta t \left(F_j(t_n,Y_j)-F_j(t_{n-1},U_{n-1})\right),
\quad j=1,2,\ldots,k, \\\\
\widetilde{Y}_0 = Y_0+ \frac{1}{2} \Delta t \left(F_0(t_n,Y_k)-F_0(t_{n-1},U_{n-1})\right),\\\\
\widetilde{Y}_j = \widetilde{Y}_{j-1}+\theta\Delta t \,(F_j(t_n,\widetilde{Y}_j)-F_j(t_{n-1},U_{n-1})),
\quad j=1,2,\ldots,k, \\\\
U_n = \widetilde{Y}_k
\end{array}\right.
\end{equation}
\\\\
\noindent
\textit{Modified Craig--Sneyd (MCS) scheme}
\begin{equation}\label{MCS}
\left\{\begin{array}{l}
Y_0 = U_{n-1}+\Delta t F(t_{n-1},U_{n-1}), \\\\
Y_j = Y_{j-1}+\theta\Delta t \left(F_j(t_n,Y_j)-F_j(t_{n-1},U_{n-1})\right),
\quad j=1,2,\ldots,k, \\\\
\widehat{Y}_0 = Y_0+ \theta \Delta t \left(F_0(t_n,Y_k)-F_0(t_{n-1},U_{n-1})\right),\\\\
\widetilde{Y}_0 = \widehat{Y}_0+ (\frac{1}{2}-\theta )\Delta t \left(F(t_n,Y_k)-F(t_{n-1},U_{n-1})\right),\\\\
\widetilde{Y}_j = \widetilde{Y}_{j-1}+\theta\Delta t \,(F_j(t_n,\widetilde{Y}_j)-F_j(t_{n-1},U_{n-1})),
\quad j=1,2,\ldots,k, \\\\
U_n = \widetilde{Y}_k
\end{array}\right.
\end{equation}
\\\\
\noindent
\textit{Hundsdorfer--Verwer (HV) scheme}
\begin{equation}\label{HV}
\left\{\begin{array}{l}
Y_0 = U_{n-1}+\Delta t F(t_{n-1},U_{n-1}), \\\\
Y_j = Y_{j-1}+\theta\Delta t \left(F_j(t_n,Y_j)-F_j(t_{n-1},U_{n-1})\right),
\quad j=1,2,\ldots,k, \\\\
\widetilde{Y}_0 = Y_0+\frac{1}{2}\Delta t \left(F(t_n,Y_k)-F(t_{n-1},U_{n-1})\right),\\\\
\widetilde{Y}_j = \widetilde{Y}_{j-1}+\theta\Delta t\,(F_j(t_n,\widetilde{Y}_j)-F_j(t_n,Y_k)),
\quad j=1,2,\ldots,k, \\\\
U_n = \widetilde{Y}_k.
\end{array}\right.
\end{equation}
\vskip0.5cm
A study of the above four ADI schemes shows that the $A_0$ part,
representing all mixed derivative terms, is always treated in an
{\it explicit}\, fashion.
In line with the original ADI idea, the $A_j$ parts for
$j=1,2,\ldots,k$ are successively treated in an {\it implicit}\,
fashion.

In its general form (\ref{Do}) the Douglas scheme has been considered
e.g.~in \cite{VHV79, HV03}.
Special instances include the well-known ADI schemes of Douglas \&
Rachford \cite{DR56}, where $F_0=0$ and $\theta=1$, and of Brian
\cite{Br61} and Douglas \cite{D62}, where $F_0=0$ and
$\theta=\frac{1}{2}$.
McKee \& Mitchell \cite{MM70} first proposed the Do scheme for
equations (\ref{PDE}) where a mixed derivative term is present,
with $k=2$.

It is readily shown that if $A_0$ is nonzero then the classical
order of consistency\footnote{That is, the order for fixed,
nonstiff ODEs.} of the Do scheme is just one, for any given
$\theta$.

The CS, MCS and HV schemes can be viewed as different extensions
to the Do scheme. They each perform a second explicit prediction
stage followed by $k$ implicit correction stages.
The CS scheme was introduced in \cite{CS88} and attains a classical
order equal to two (independently of $A_0$) if $\theta=\frac{1}{2}$.
The MCS scheme was defined in \cite{IHW09} and possesses a classical
order equal to two for any given $\theta$.
Note that the CS and MCS schemes are identical if $\theta=\frac{1}{2}$.
The HV scheme was constructed in \cite{Hun02} and was proposed for
equations with mixed derivative terms in \cite{IHW07}.
Like the MCS scheme, the HV scheme possesses the favorable property
of having a classical order equal to two for any given $\theta$.

For the stability analysis in this paper the following linear scalar
{\it test equation}\, is relevant,
\begin{equation}\label{test}
U'(t) = (\lambda_0+\lambda_1+\cdots+\lambda_k)U(t),
\end{equation}
with complex constants  $\lambda_j$ ($0\le j\le k$).
Application of any given ADI scheme in the case of test equation
(\ref{test}) gives rise to a linear iteration of the form
\[
U_n = M(z_0,z_1,\ldots,z_k)\,U_{n-1}\,.
\]
Here $M$ is a given, fixed, multivariate rational function and
$z_j = \Delta t\, \lambda_j$ ($0\le j\le k$).
The iteration is stable if
\begin{equation}\label{stab}
|M(z_0,z_1,\ldots,z_k)|\le 1.
\end{equation}
Write
\begin{equation}\label{zp}
z = z_1+z_2+\cdots+z_k \quad\text{and}\quad p = (1-\theta z_1)
(1-\theta z_2)\cdots(1-\theta z_k).
\end{equation}
For the schemes (\ref{Do}), (\ref{CS}), (\ref{MCS}), (\ref{HV}) it
is readily verified that $M=R,\,\widetilde{S},\,S,\,T$, respectively,
where
\begin{eqnarray}
R(z_0,z_1,\ldots,z_k)&=&1+\frac{z_0+z}{p}\label{R}\,,\\
\widetilde{S}(z_0,z_1,\ldots,z_k)&=&\label{S1}
1+\frac{z_0+z}{p}+\frac{1}{2}\, \frac{z_0(z_0+z)}{p^2}\,,\\
S(z_0,z_1,\ldots,z_k)&=&\label{S2}
1+\frac{z_0+z}{p}+\theta\,
\frac{z_0(z_0+z)}{p^2}+\left(\tfrac{1}{2}-\theta\right)\frac{(z_0+z)^2}{p^2}\,,\\
T(z_0,z_1,\ldots,z_k)&=&\label{T}
1+2\, \frac{z_0+z}{p}-\frac{z_0+z}{p^2}+\frac{1}{2}\frac{(z_0+z)^2}{p^2}\,.
\end{eqnarray}

\subsection{Finite difference discretization}\label{FD}
For the semidiscretization of (\ref{PDE}) we replace all spatial
derivatives by second-order central finite differences on a
rectangular grid with constant mesh width $\Delta x_i>0$ in the
$x_i$--direction ($1\le i\le k$):
\begin{subeqnarray}\label{discretization}
\left(u_{x_ix_i}\right)_\ell
&\approx & \frac{u_{\ell+e_i}-2u_\ell+u_{\ell-e_i}}{(\Delta x_i)^2},\\
\left(u_{x_ix_j}\right)_\ell \notag
&\approx & \frac{(1+\beta_{ij})(u_{\ell+e_i+e_j}+u_{\ell-e_i-e_j})
-(1-\beta_{ij})(u_{\ell-e_i+e_j}+u_{\ell+e_i-e_j})}{4\Delta x_i \Delta x_j}\\
&&+\frac{\beta_{ij}(4u_\ell-2(u_{\ell+e_i}+u_{\ell+e_j}+u_{\ell-e_i}+u_{\ell-e_j}))}
{4\Delta x_i \Delta x_j},\quad i\ne j.
\end{subeqnarray}
Here $\ell=(\ell_1,\ell_2,\ldots,\ell_k)$ and the unit vectors
$e_1,e_2,\ldots,e_k$ denote multi-indices, $u_\ell$ stands for
$u(\ell_1\Delta x_1,\ell_2\Delta x_2,\ldots,\ell_k\Delta x_k,t)$
and $\beta_{ij}$ ($1\le i\not= j \le k$) are given real parameters
with $\beta_{ij}=\beta_{ji}$.
The righthand side of (\ref{discretization}b) is the most general
second-order finite difference formula for the mixed derivative
$u_{x_ix_j}$ that is based on a centered 9-point stencil.
If $\beta_{ij} = 0$, then (\ref{discretization}b) reduces to the
standard 4-point formula
\[ \left(u_{x_ix_j}\right)_\ell
\approx \frac{u_{\ell+e_i+e_j}+u_{\ell-e_i-e_j}-u_{\ell-e_i+e_j}-u_{\ell+e_i-e_j}}
{4\Delta x_i \Delta x_j}.
\]
In the literature also the choices $\beta_{ij} = -1$ and $\beta_{ij} = 1$
are frequently considered.

\subsection{Stability analysis and outline of this paper}\label{StabAnal}
The aim of this paper is to investigate the stability of the four
ADI schemes (\ref{Do})--(\ref{HV}) in the application to semidiscretized
multidimensional diffusion equations (\ref{PDE}) where the condition
(\ref{gamma}) on the size of the mixed derivative coefficients
is effectively taken into account.
Our stability analysis is equivalent to the well-known von Neumann
(Fourier) analysis.
Accordingly, (\ref{PDE}) is considered with constant diffusion matrix
$D$ and periodic boundary condition and stability pertains to the
$l_2$-norm.
The obtained semidiscrete matrices $A_j$ ($0\le j\le k$) are then all
normal and commuting and can therefore be unitarily diagonalized by a
single matrix.
For any given ADI scheme, one thus arrives at the stability requirement
(\ref{stab}) with $\lambda_j$ eigenvalues of the matrices $A_j$
($0\le j\le k$).

Upon substituting discrete Fourier modes into (\ref{discretization}),
it is readily shown that the scaled eigenvalues $z_j$ are given by
\begin{subeqnarray}\label{eig}
z_0&=&\sum_{i\ne j} r_{ij}\,d_{ij}\,[-\sin\phi_i\sin\phi_j+\beta_{ij}
(1-\cos\phi_i)(1-\cos\phi_j)]\, ,\phantom{xxx}\\
z_j &=&-2r_{jj}d_{jj}(1-\cos\phi_j), \quad j=1,2,\ldots,k.
\end{subeqnarray}
The angles $\phi_j$ are integer multiples of $2\pi/m_j$ with $m_j$
the dimension of the grid in the $x_j$--direction ($1\le j\le k$).
Further,
\[
r_{ij}=\frac{\Delta t}{\Delta x_i\Delta x_j}\quad
\text{for}~~1\le i,j\le k.
\]

Stability results for ADI schemes pertinent to (\ref{gamma}) in
the case of $\gamma= 1$ were derived by Craig \& Sneyd~\cite{CS88}
for the Do and CS schemes and by In 't Hout \& Welfert \cite{IHW09}
for the MCS and HV schemes.
For any given scheme, both necessary and sufficient conditions
on its parameter~$\theta$ were obtained such that the stability
requirement (\ref{stab}) with $(z_0, z_1,\ldots,z_k)$ given by
(\ref{eig}) is unconditionally fulfilled, i.e., for all
$\Delta t >0$ and all $\Delta x_i >0$ ($1\le i\le k$).

The stability analysis of ADI schemes based on (\ref{gamma})
with {\it arbitrary}\, given $\gamma \in [0,1]$ was recently
started in In 't Hout \& Mishra \cite{IHM10}.
For the MCS scheme and $k=2$, the useful result was proved
here that the lower bound
$\theta \ge \max \{\tfrac{1}{4}, \tfrac{1}{6}(\gamma +1) \}$
is sufficient for unconditional stability
(whenever $0\le \gamma\le 1$).

The present paper substantially extends the work of
\cite{CS88,IHM10,IHW09} reviewed above.
Section~\ref{main}~contains the two main results of this paper.
The first main result is Theorem~\ref{theorem1}, which
provides for each of the Do, CS, MCS, HV schemes in
$k=2$ and $k=3$ spatial dimensions a {\it sufficient}\, condition
on the parameter $\theta$ for unconditional stability under
(\ref{gamma}) for arbitrary given $\gamma\in [0,1]$.
The second main result is Theorem~\ref{theorem2}. This yields for
any of the four ADI schemes, any spatial dimension $k\ge 2$ and
any $\gamma \in [0,1]$ a {\it necessary}\, condition on $\theta$
for unconditional stability under (\ref{gamma}).
For each scheme, the obtained necessary and sufficient conditions
coincide if $k=2$ or $k=3$.
Section~\ref{numexp} presents numerical illustrations to the two
main theorems.
The final Section \ref{concl} gives conclusions and issues for
future research.

\setcounter{equation}{0}
\setcounter{theorem}{0}
\section{Main results}\label{main}
In the following we always make the minor assumption that the matrix
$B=(-\beta_{ij})_{1\le i,j\le k}$ with $\beta_{ii}=-1$ ($1\le i\le k$)
is positive semidefinite.
Thus, in particular, $|\beta_{ij}|\le 1$ for all $i, j$.
We note that this assumption on $B$ is weaker than the corresponding
assumption that was made in \cite{IHW09}.

\subsection{Preliminaries}
This section gives two lemmas that shall be used in the proofs
of the main results below.
\begin{lemma}\label{lemma2}
Let $\alpha$, $\delta$ be given real numbers with $0<\delta\le 4$.
Consider the polynomial
\[
P(u,v,w)= \alpha+u^2+v^2+w^2+uvw-\delta(u+v+w)\quad {\it for}~u, v, w \in \R.
\]
Then
\[
P(u,v,w) \ge 0 \quad{\it whenever}~u\ge 0,\, v\ge 0,\, w\ge 0
\]
if and only if
\begin{equation}\label{pol}
(\delta +1)(3-2\sqrt{\delta+1})\ge 1-\alpha ~~{\it and}~~ \delta^2 \le 2\alpha.
\end{equation}
\end{lemma}
\begin{proof}
The critical points $(u,v,w)$ of $P$ are given by the equations
\[
2u+vw = 2v+uw = 2w+uv = \delta.
\]
A straightforward analysis, using $0<\delta\le 4$, shows that there
is precisely one critical point in the domain $u>0,\, v>0,\, w>0$,
and it is given by $u=v=w=\sqrt{\delta+1}-1 =: u^*$.
Inserting this into $P$ and rewriting yields
\begin{eqnarray*}
P(u^*, u^*, u^*)
& = & (u^*+1)^3 - 3\left(\delta+1\right)u^* + \alpha-1\\
& = & \left(\sqrt{\delta+1}\right)^3 - 3\left(\delta+1\right)
      \left(\sqrt{\delta+1} -1\right) + \alpha-1\\
& = & \left(\delta+1\right)\left(3-2\sqrt{\delta+1}\right)+\alpha-1.
\end{eqnarray*}
Hence $P(u^*, u^*, u^*)\ge 0$ if and only if the first inequality
in (\ref{pol}) holds.

Consider next the polynomial $P$ on the boundary of the pertinent
domain.
It is clear that $P(u,v,w) \ge 0$ if $u, v, w \rightarrow\infty$.
Next, on the boundary part $u\ge 0$, $v\ge 0$, $w=0$ there holds
\begin{eqnarray*}
P(u,v,0)
& = & \alpha+u^2+v^2 - \delta(u+v) \\
& = & (u-\tfrac{1}{2}\delta)^2 + (v-\tfrac{1}{2}\delta)^2
      + \alpha - \tfrac{1}{2}\delta^2.
\end{eqnarray*}
Thus $P(u,v,0)\ge 0$ (whenever $u\ge 0$, $v\ge 0$) if and only if the
second inequality in (\ref{pol}) holds.
By symmetry, the result for the other two boundary parts is the same,
which completes the proof.
\begin{flushright}
\mbox{\tiny $\blacksquare$}
\end{flushright}
\end{proof}
\vspace*{10pt}

The subsequent analysis relies upon four key properties of the scaled
eigenvalues (\ref{eig}).
\begin{lemma}\label{lemma1}
Let $z_0, z_1,\ldots, z_k$ be given by (\ref{eig}).
Let $\gamma\in [0,1]$ and assume (\ref{gamma}) holds.
Then:
\begin{subeqnarray}\label{zproperties}
&&all~z_j~are~real,\\
&&z_j\le 0 ~~{\it for}~1\le j\le k,\phantom{xxxxxxxxxxxxx}\\
&&z+z_0 \le 0,\\
&&|z_0| \le \gamma \sum_{i\ne j} \sqrt{z_i z_j}\,.
\end{subeqnarray}
\end{lemma}
\begin{proof}
Properties (\ref{zproperties}a), (\ref{zproperties}b)
are obvious.
We consider (\ref{zproperties}c) and (\ref{zproperties}d).
Using $r_{ij}=\sqrt{r_{ii}r_{jj}}$ and the simple identity
\[
(\sin\phi)^2+(1-\cos\phi)^2 = 2(1-\cos\phi)
\]
one readily verifies that
\[
-(z+z_0) = \sum_{i,j=1}^k r_{ij}\,d_{ij}\,[\sin\phi_i\sin\phi_j
-\beta_{ij}(1-\cos\phi_i)(1-\cos\phi_j)]
= \mathbf{u}^{\rm T}D\mathbf{u}+\mathbf{v}^{\rm T}(D\circ B)\mathbf{v},
\]
where $\circ$ denotes the Hadamard product of two matrices and
\begin{equation*}
\mathbf{u}=\left(\,\sqrt{r_{jj}}\,\sin\phi_j\,\right)_{1\le j\le k}
~~,~~
\mathbf{v}=\left(\,\sqrt{r_{jj}}\,(1-\cos\phi_j)\,\right)_{1\le j\le k}.
\end{equation*}
It is well-known that the Hadamard product of two positive semidefinite
matrices is also positive semidefinite, see e.g. \cite{HJ85}, and hence,
$z+z_0\le 0$ is obtained.

Next, using the Cauchy--Schwarz inequality and $|\beta_{ij}|\le 1$,
it directly follows that
\[
|\sin\phi_i\sin\phi_j-\beta_{ij}(1-\cos\phi_i)(1-\cos\phi_j)|\le
\sqrt{2(1-\cos\phi_i)}\cdot \sqrt{2(1-\cos\phi_j)}.
\]
Together with (\ref{gamma}) this yields
\[
|z_0|\le \gamma\sum_{i\ne j} \sqrt{r_{ii}r_{jj}}\,\sqrt{d_{ii}d_{jj}}\,
\sqrt{2(1-\cos\phi_i)}\cdot\sqrt{2(1-\cos\phi_j)} = \gamma\sum_{i\ne j}
\sqrt{z_i z_j}\,.
\]
\begin{flushright}
\mbox{\tiny $\blacksquare$}
\end{flushright}
\end{proof}
\eject

\noindent
In the following we shall use the notation $y_j = \sqrt{-\theta z_j}$
(for $1\le j\le k$).
Then
\begin{equation}\label{defyj}
   p = \prod_{j=1}^k (1+y_j^2)\ge 1 \quad \text{and}
\quad z=-\frac{1}{\theta}\sum_{j=1}^k y_j^2\\
\end{equation}
and by (\ref{zproperties}d) there holds
\begin{equation}\label{z0plusz}
    z+z_0 \geq z-|z_0| \geq \sum_{j=1}^k z_j - \gamma\sum_{i\neq j}\sqrt{z_iz_j} =
    -\frac{1}{\theta}\left(\sum_{j=1}^k y_j^2+\gamma\sum_{i\neq j} y_iy_j\right).
\end{equation}

\subsection{Two main theorems}

The first main result gives sufficient conditions on $\theta$ for
unconditional stability of each of the Do, CS, MCS, HV schemes in
the application to (\ref{PDE}) for $k=2$ or $k=3$ under condition
(\ref{gamma}) with arbitrary given $\gamma$.

\begin{theorem}\label{theorem1}
Consider equation (\ref{PDE}) for $k=2$ or $k=3$ with symmetric
positive semidefinite matrix $D$ and periodic boundary condition.
Let $\gamma \in [0,1]$ and assume (\ref{gamma}) holds. Let
(\ref{ODE}), (\ref{splitting}) be obtained by central second-order
FD discretization and splitting as described in Section~\ref{intro}.
Then for the following parameter values $\theta$ the Do, CS, MCS,
HV schemes are unconditionally stable when applied to (\ref{ODE}),
(\ref{splitting}):\\\\
\noindent
$~~~~\bullet$ Do scheme:
\[
\theta \geq\frac{1}{2} ~(if~k=2) ~~~and~~~
\theta \geq\max\left\{\frac{1}{2}, \frac{2(2\gamma+1)}{9}\right\} ~(if~k=3);
~~~~~~~~~~~~~~~~~~~
\]
\noindent
$~~~~\bullet$ CS scheme:
\[
\theta \geq\frac{1}{2} ~(if~k=2~or~k=3);
\phantom{\max\left\{\frac{1}{2}, \frac{2(2\gamma+1)}{9}\right\}}
~~~~~~~~~~~~~~~~~~~~~~~~~~~~~~~~~~~~~~
\]
\noindent
$~~~\bullet$ MCS scheme:
\[
\theta \geq\max\left\{\frac{1}{4}, \frac{\gamma+1}{6}\right\} ~(if~k=2) ~~~and~~~
\theta \geq\max\left\{\frac{1}{4}, \frac{2(2\gamma+1)}{13}\right\} ~(if~k=3);
\]
\noindent
$~~~\bullet$ HV scheme:
\[
~~\theta \geq\max\left\{\frac{1}{4}, \frac{\gamma+1}{4+2\sqrt 2}\right\}  ~(if~k=2)
~~~and~~~
\theta \geq\max\left\{\frac{1}{4}, \frac{2\gamma+1}{4+2\sqrt 3}\right\} ~(if~k=3).
\]
\end{theorem}

\begin{proof}
The proofs for the four schemes are similar.
In view of this, we shall consider here the HV scheme and leave the
proofs for the Do, CS and MCS schemes to the reader.\footnote{See,
however, \cite{IHM10} for the MCS scheme if $k=2$.}

Using the properties (\ref{zproperties}a)--(\ref{zproperties}c) of
the scaled eigenvalues, it is readily seen that for $M=T$ the
stability requirement (\ref{stab}) is equivalent to
\begin{subeqnarray}
\label{condHV}
&&2p^2+(2p-1)(z_0+z)+\tfrac{1}{2}(z_0+z)^2 \geq 0 \,,\\
&&2p-1+\tfrac{1}{2}(z_0+z) \geq 0\,.
\end{subeqnarray}
Condition (\ref{condHV}a) is always fulfilled since the discriminant
$(2p-1)^2 - 4p^2 = -4p+1 <0$.
Subsequently, using (\ref{defyj}) and (\ref{z0plusz}) it is easily
seen that if there exists $\kappa > 0$ such that
\begin{equation}
\label{kappa}
2\prod_{j=1}^k (1+y_j^2) - 1 - \kappa\left(\sum_{j=1}^k y_j^2 +\gamma\sum_{i\neq j}y_iy_j\right)\geq 0
\end{equation}
for all $y_j \ge 0$ ($1\le j \le k$), then condition (\ref{condHV}b)
is fulfilled whenever $\theta\geq 1/(2\kappa)$.
\\\\
\noindent
\underline{For $k=2$}\, the inequality (\ref{kappa}) reads
\[
1+2(y_1^2+y_2^2+y_1^2y_2^2) -\kappa(y_1^2+y_2^2+2\gamma y_1y_2)\geq 0,
\]
and this can be rewritten as
\[
(2-\kappa)(y_1-y_2)^2 + 2\left(y_1y_2
+ 1-\tfrac{1}{2}\kappa-\tfrac{1}{2}\kappa\gamma\right)^2 + 1
-2\left(1-\tfrac{1}{2}\kappa-\tfrac{1}{2}\kappa\gamma\right)^2
\geq 0 \, .
\]
Hence, the inequality is satisfied if
\[
\kappa\leq 2 \quad \text{and} \quad \sqrt{2} \left| 1-\tfrac{1}{2}\kappa-\tfrac{1}{2}\kappa\gamma\right|
\leq 1 \, ,
\]
which is equivalent to
\[
\frac{2-\sqrt 2}{\gamma+1} \le \kappa\leq \min\left\{2, \, \frac{2+\sqrt 2}{\gamma+1} \right\}.
\]
Selecting the rightmost value for $\kappa$, yields that (\ref{condHV}b)
holds for $k=2$ whenever
\[
\theta\geq\max\left\{ \frac{1}{4}, \frac{\gamma+1}{4+2\sqrt{2}}\right\}.
\]
\noindent
\underline{For $k=3$}\, the inequality (\ref{kappa}) reads
\[
1+(2-\kappa)\sum_{j=1}^3y_j^2 + 2\sum_{i<j}y_i^2 y_j^2 + 2y_1^2y_2^2y_3^2
-2\kappa\gamma\sum_{i<j}y_iy_j \ge 0,
\]
which, by the identity
\[
\sum_{i < j} (y_i-y_j)^2 = 2 \sum_{j=1}^3y_j^2 - 2 \sum_{i<j}y_i y_j,
\]
is equivalent to
\[
1+\tfrac{1}{2}(2-\kappa)\sum_{i < j} (y_i-y_j)^2 + 2\sum_{i<j}y_i^2 y_j^2 + 2y_1^2y_2^2y_3^2
+(2-\kappa-2\kappa\gamma)\sum_{i<j}y_iy_j \ge 0.
\]
Let $u=y_1y_2$, $v=y_1y_3$, $w=y_2y_3$.
Then $u$, $v$, $w\ge 0$ and the above inequality can be written as
\[
\tfrac{1}{2}(2-\kappa)\sum_{i < j} (y_i-y_j)^2 + 2P(u,v,w) \ge 0,
\]
where
\[
P(u,v,w) = \alpha + u^2+v^2+w^2 + uvw -\delta(u+v+w)
\]
with
\[
\alpha = \tfrac{1}{2}~~~{\rm and}~~~\delta = \tfrac{1}{2}\kappa+\kappa\gamma-1.
\]
Hence, if $\kappa\le 2$ and $P(u,v,w)\ge 0$, then (\ref{kappa}) is
satisfied for $k=3$.
If $\delta\le 0$, i.e.~$\kappa\le 2/(2\gamma+1)$, then obviously
$P(u,v,w)\ge \alpha > 0$.
Next consider $\delta>0$.
Lemma \ref{lemma2} yields that $P(u,v,w)\ge 0$ whenever
\begin{equation}\label{beta}
(\delta +1)(3-2\sqrt{\delta+1})\ge \tfrac{1}{2} ~~{\rm and}~~ \delta \le 1.
\end{equation}
Set $x=\sqrt{\delta+1}$. Then $x>1$ and
\[
(\delta +1)(3-2\sqrt{\delta+1})\ge \tfrac{1}{2} ~\Longleftrightarrow~ 4x^3-6x^2+1 \le 0
~\Longleftrightarrow~ (2x-1)(2x^2-2x-1) \le 0 ~\Longleftrightarrow~ x\le \tfrac{1}{2}+\tfrac{1}{2}\sqrt{3}.
\]
It follows that (\ref{beta}) is equivalent to $\delta \le \tfrac{1}{2}\sqrt{3}$,
which means $\kappa\le (2+\sqrt{3})/(2\gamma+1)$.
Hence, if
\[
\kappa\leq \min\left\{2, \, \frac{2+\sqrt{3}}{2\gamma+1} \right\},
\]
then (\ref{kappa}) holds for $k=3$.
Selecting the upper bound for $\kappa$, yields that (\ref{condHV}b)
is fulfilled for $k=3$ whenever
\[
\theta \geq\max\left\{\frac{1}{4}, \frac{2\gamma+1}{4+2\sqrt{3}}\right\}.
\]
\begin{flushright}
\mbox{\tiny $\blacksquare$}
\end{flushright}
\end{proof}

Upon setting $\gamma=1$ in Theorem~\ref{theorem1} the resulting
sufficient conditions on $\theta$ for the CS, MCS, HV schemes
agree with those in \cite[Thms.~2.2, 2.5, 2.8]{IHW09}.
The above theorem thus forms a proper generalization of results
in \cite{IHW09}.
For the Do scheme, the obtained sufficient condition generalizes and
improves the corresponding result for this scheme from \cite{CS88}.
In particular, if $\gamma=1$ and $k=3$, then \cite{CS88} yields
$\theta\ge 3\sqrt{3} - \frac{9}{2}\approx 0.696$, whereas
Theorem~\ref{theorem1} yields $\theta\ge \frac{2}{3}$.

In view of Theorem~\ref{theorem1}, a smaller parameter value
$\theta$ can be chosen while retaining unconditional stability
if it is known that (\ref{gamma}) holds with certain given
$\gamma<1$.
This is useful, in particular since a smaller value $\theta$
often yields a smaller, i.e., more favorable, error constant.

The MCS scheme with the lower bound for $\theta$ given
by Theorem~\ref{theorem1} has been successfully used recently in the
actual application to the three-dimensional Heston--Hull--White PDE
from financial mathematics, see Haentjens \& In 't Hout \cite{HH12}.

The following theorem provides, for each ADI scheme, a necessary
condition on $\theta$ for unconditional stability.

\begin{theorem}\label{theorem2}
Let $k\ge 2$ and $\gamma \in [0,1]$ be given and consider any ADI
scheme from Section~\ref{intro}.
Suppose that the scheme is unconditionally stable whenever it is
applied to a system (\ref{ODE}), (\ref{splitting}) that is obtained
by central second-order FD discretization and splitting as described
in Section~\ref{intro} of any equation (\ref{PDE}) with positive
semidefinite $k\times k$ matrix $D$ satisfying (\ref{gamma}) and
periodic boundary condition.
Then $\theta$ must satisfy the following bound:\\\\
\noindent
$~~~~\bullet$ Do scheme:
\[
\theta \ge \max\left\{\frac{1}{2}\,,\, \frac{1}{2}\,d_k\,[(k-1)\gamma+1] \right\}
\quad with \quad
d_k=\left(1-\frac{1}{k}\right)^{k-1}
~~~~~~
\]
\noindent
$~~~\bullet$ CS scheme:
\[
\theta \ge \max\left\{\frac{1}{2}\,,\, \frac{1}{2}\,c_k\,k\gamma \right\}
\quad with \quad
c_k=\left(1-\frac{1}{k}\right)^{k}
~~~~~~~~~~~~~~~~~~~~~~~
\]
\noindent
$~~~\bullet$ MCS scheme:
\[
\theta \ge \max\left\{\frac{1}{4}\,,\, \frac{1}{2}\,b_k\,[(k-1)\gamma+1] \right\}
\quad with \quad
b_k=\frac{1}{1+\left(1+\frac{1}{k-1}\right)^{k-1}}
\]
\noindent
$~~~\bullet$ HV scheme:
\[
\theta \ge \max\left\{\frac{1}{4}\,,\, \frac{1}{2}\,a_k\,[(k-1)\gamma+1] \right\}
~~~~~~~~~~~~~~~~~~~~~~~~~~~~~~~~~~~~~~~~
\]
\begin{equation*}\label{alphastar}
\qquad ~~with~a_k~the~unique~solution~a\in\left(0\,,\,\frac{1}{2}\,\right)
~of~~2a\left(1+\frac{1-a}{k-1}\right)^{k-1}\!\!\!\!-1=0.
\end{equation*}
\end{theorem}
\begin{proof}
As in the foregoing proof, we consider here the HV scheme and
leave the (analogous) proofs for the other three ADI schemes
to the reader.

Consider the $k\times k$ matrix $D=(d_{ij})$ with $d_{ii}=1$ and
$d_{ij} = \gamma$ whenever $1\le i\not= j \le k$.
Clearly, $D$ is symmetric positive semidefinite and (\ref{gamma})
holds.
Let $\Delta x_i \equiv \Delta x >0$, so that
$r_{ij} \equiv r:=\Delta t /(\Delta x)^2$.

First, choose the angles $\phi_j$ in (\ref{eig}) equal to zero for
$j\ge 2$ .
Then the eigenvalues $z_j$ are given by $z_0=0$, $z_1=-2r(1-\cos\phi_1)$
and $z_2=z_3=\ldots=z_k=0$.
In view of (\ref{condHV}b), we have
\[ 2p-1+\tfrac{1}{2}(z_0+z) = 1+(\tfrac{1}{2}-2\theta)z_1 \geq 0 \]
whenever $r>0$, $\phi_1\in [0,2\pi)$.
This immediately implies $\theta \ge \tfrac{1}{4}$.

Next, choose all angles $\phi_j$ in (\ref{eig}) the same, i.e.,
$\phi_j \equiv \phi$.
Then the eigenvalues $z_j$ are given by
\begin{subeqnarray*}
z_0&=&-rk(k-1)\gamma [\, \sin^2\phi - \overline{\beta}\,
(1-\cos\phi)^2\, ]\, ,\phantom{\sum}\\
z_j &=&-2r(1-\cos\phi), \quad j=1,2,\ldots,k,
\end{subeqnarray*}
where
\[
\overline{\beta} = \frac{\sum_{i \not= j} \beta_{ij}}{k(k-1)}\,.
\]
Note that $|\overline{\beta}| \le 1$.
By (\ref{condHV}b), it must hold that
\[
\theta \ge - \frac{1}{2} \frac{\theta z_0+\theta z}{2p-1} =
\frac{k}{2} \, \frac{(k-1)\gamma \theta r [\, \sin^2\phi - \overline{\beta}\,
(1-\cos\phi)^2\, ]+2\theta r (1-\cos\phi)}
{2[\, 1+2\theta r (1-\cos\phi)\,]^k-1}
\]
whenever $r>0$, $\phi\in [0,2\pi)$.
Setting $\alpha = 2\theta r(1-\cos\phi)$, this yields
\[
\theta \ge \frac{k}{2} \frac{(k-1)\gamma \alpha [\,1+\cos\phi
 -\overline{\beta}\, (1-\cos\phi)\, ]/2+\alpha}{2(1+\alpha)^k-1}
\]
for all $\alpha >0$, $\phi\in (0,2\pi)$.
Taking the supremum over $\phi$, gives
\[
\theta \ge \frac{1}{2}\, h(\alpha)\,[(k-1)\gamma+1]
\quad {\rm for~all}~~\alpha >0,
\]
where
\[
h(\alpha) = \frac{k\alpha}{2(1+\alpha)^k-1}.
\]
By elementary arguments (cf.~\cite{IHW09}) it follows that the
function $h$ has an absolute maximum which is equal to $a_k$
given in the theorem.
\begin{flushright}
\mbox{\tiny $\blacksquare$}
\end{flushright}
\end{proof}

Upon setting $\gamma=1$ in Theorem~\ref{theorem2}, the necessary
conditions on $\theta$ for the CS, MCS, HV schemes reduce to
those given in \cite[Thms.~2.3, 2.6, 2.9]{IHW09}.
Further, for the Do scheme and $\gamma=1$ there is agreement
with the necessary condition from \cite{CS88}.


It is readily verified that, for each ADI scheme, the sufficient
conditions of Theorem~\ref{theorem1} and the necessary conditions
of Theorem~\ref{theorem2} are identical whenever $k=2$ or $k=3$
and $0\le \gamma \le 1$.
Hence, in two and three spatial dimensions, these conditions are
sharp.

The interesting question arises whether the necessary conditions of
Theorem~\ref{theorem2} are also sufficient in spatial dimensions
$k\ge 4$.
In \cite{IHW09} it was proved that this is true for the HV scheme
and $\gamma = 1$.
It can be seen, however, that the proof from \cite{IHW09} does
not admit a straightforward extension to values $\gamma < 1$.
Further, in the case of the Do, CS, MCS schemes a proof is not
clear at present.
Accordingly, we leave this question for future research.

\setcounter{equation}{0}
\setcounter{theorem}{0}
\section{Numerical illustration}\label{numexp}
In this section we illustrate the main results of this paper,
Theorems~\ref{theorem1} and \ref{theorem2}.
We present experiments where all
four ADI schemes (\ref{Do})--(\ref{HV}) are applied in the
numerical solution of multidimensional diffusion equations
(\ref{PDE}) possessing mixed derivative terms.
The PDE is semidiscretized by central second-order finite
differences as described in Subsection \ref{FD}, with
$\beta_{ij} \equiv 0$, and the semidiscrete matrix $A$
is splitted as described in Subsection~\ref{ADI}.
The boundary condition is taken to be periodic
(in this case $g(t)\equiv 0$).

The first experiment deals with (\ref{PDE}) in two
spatial dimensions. We choose initial function
\begin{equation*}
\label{ic}
u(x_1,x_2,0) = e^{-4[\sin^2(\pi x_1)+\cos^2(\pi x_2)]}~~~
(0\le x_1, x_2 \le 1)
\end{equation*}
and diffusion matrix $D$ given by
\[
D = 0.025
\left(
  \begin{array}{cc}
    1 & 2\gamma \\
    2\gamma & 4 \\
  \end{array}
\right).
\]
This matrix $D$ is positive semidefinite and the condition
(\ref{gamma}) holds whenever $\gamma\in [0,1]$.
Consider $\gamma=0.9$.
Then the lower bounds on $\theta$ provided by
Theorem~\ref{theorem1} for the MCS and HV schemes are,
rounded to three decimal places, equal to 0.317 and 0.278,
respectively.
For the Do and CS schemes, the lower bound always
equals $\tfrac{1}{2}$ in two spatial dimensions,
independently of $\gamma$.

Figure~\ref{fig1} shows for
$\Delta x_1 = \Delta x_2 = 1/80$ the semidiscrete
solution values $U(0)$ and $U(5)$ displayed on the grid
in $\Omega$, so that they represent the exact solution
$u$ at $t=0$ and $t=5$.
\vskip0.3cm
\begin{figure}[h]
\begin{center}
\begin{tabular}{cc}
$t=0$ & $t=5$\\
\includegraphics[width=0.45\textwidth]{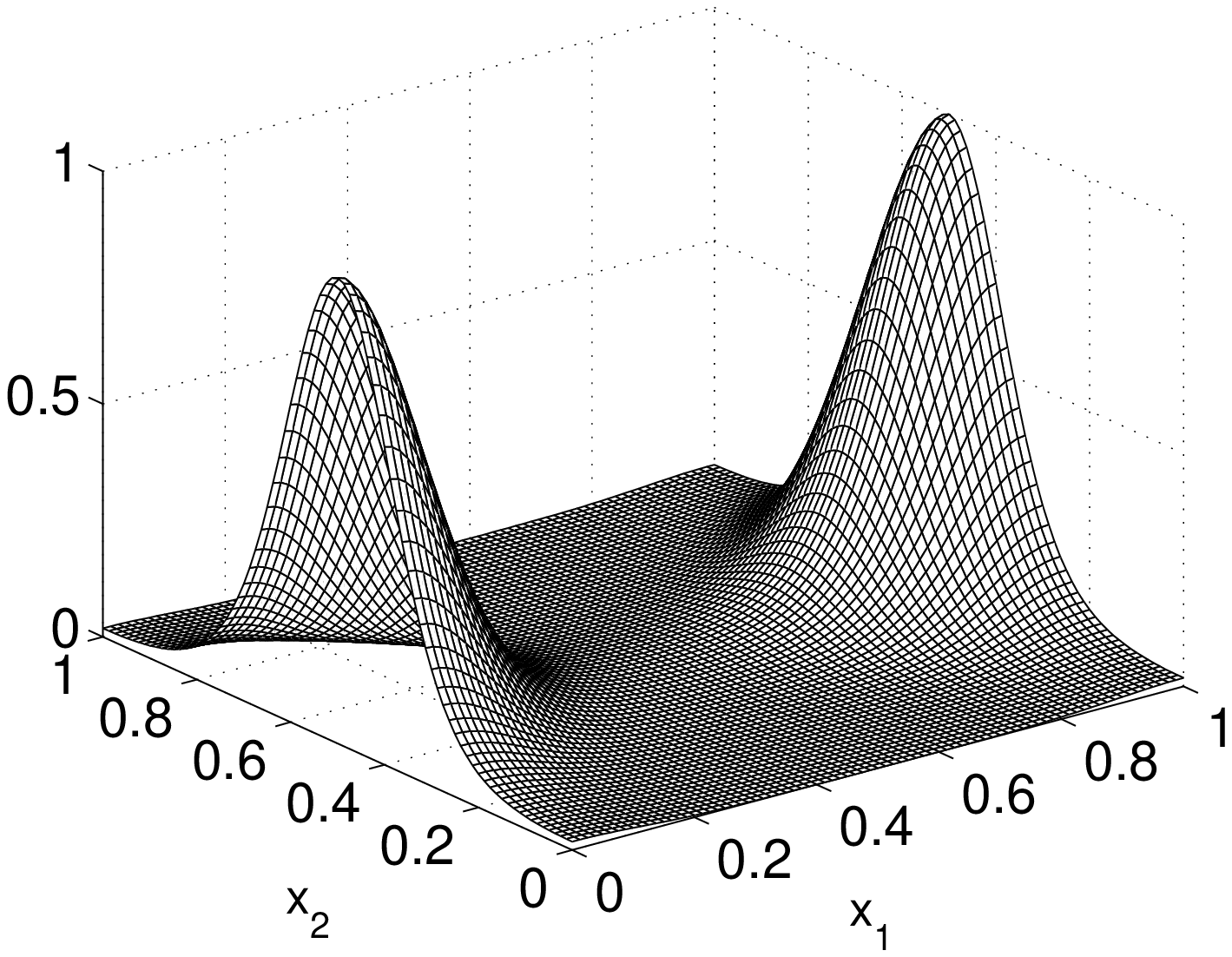}&
\includegraphics[width=0.45\textwidth]{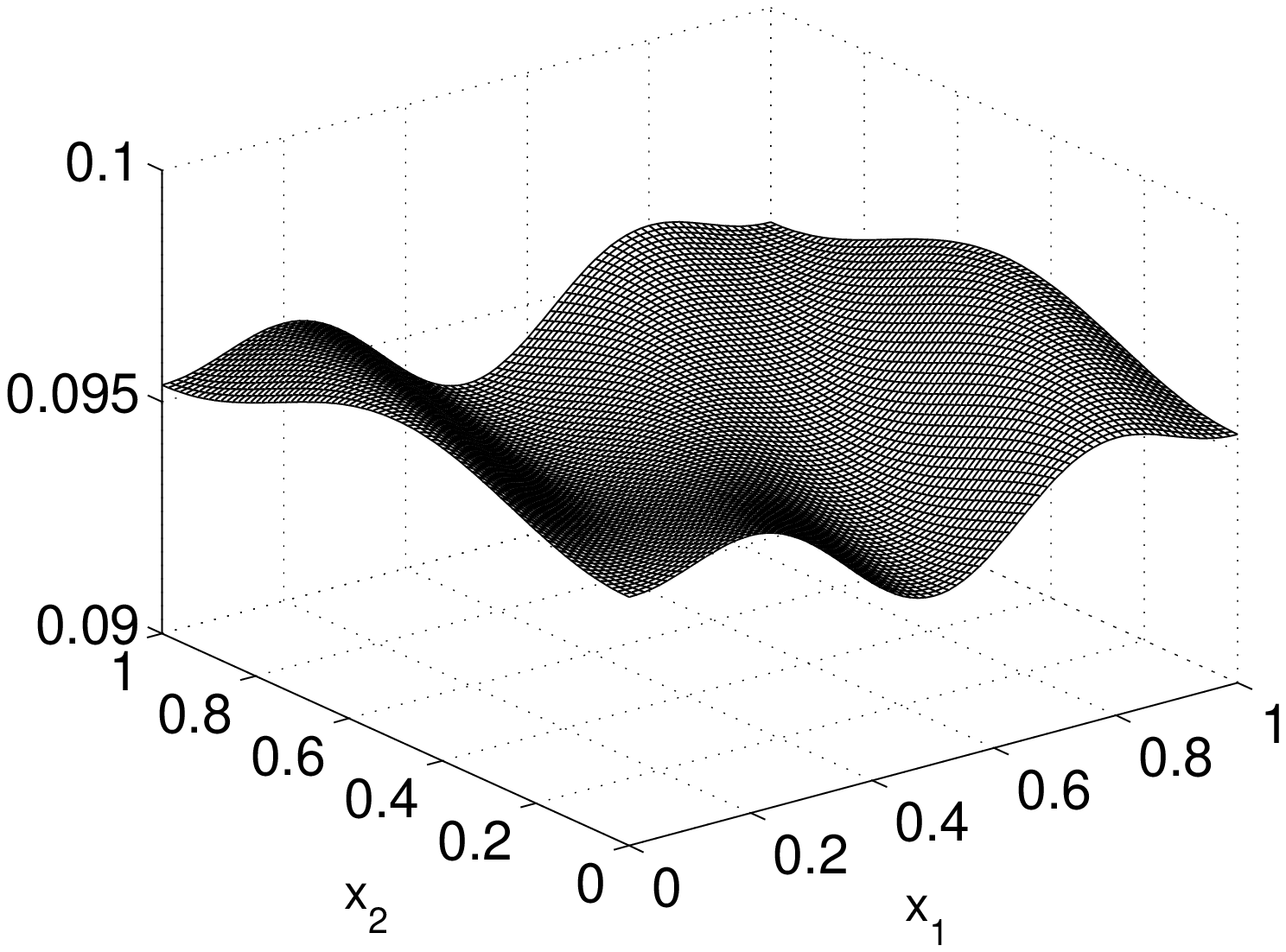}
\end{tabular}
\caption{Exact solution $u$ of 2D problem at $t=0$
and $t=5$.
Note the different scales on the vertical axes.}
\label{fig1}
\end{center}
\end{figure}
To the semidiscrete systems with $\Delta x_1 = \Delta x_2 = 1/m$
for $m=40,\, 80$ we have applied the Do, CS, MCS, HV schemes
using their exact lower bound values~$\theta$ (above) as well
as the values 0.45, 0.45, 0.29, 0.25, respectively, which are
approximately 90\% of these.
For a sequence of step sizes $\Delta t = 1/N$ with
$10^{-3} \le \Delta t \le 10^0$ we computed the {\it global
temporal errors}\, at time $t=5$, defined~by
\[
e(\Delta t; m) = m^{-1} \left| U(5)-U_{5N} \right|_2,
\]
where $|\cdot|_2$ is the Euclidean vector norm
and $1/m$ is a normalization factor.
Figure~\ref{2Derrors} displays the obtained result.

From the right column of Figure~\ref{2Derrors} it is
clear that, for each ADI scheme, using the
smaller value $\theta$ leads to large temporal errors
for natural step sizes.
We verified that these large errors correspond
to the spectral radii of the pertinent iteration
matrices being greater than 1. This indicates that
they are caused by a lack of stability.
Additional experiments with larger values $t$ (e.g.
$t=10$) also shows that the large temporal errors
are amplified, as one may expect.

The left column of Figure~\ref{2Derrors} displays the
results in the case where the lower bound values
$\theta$, given by Theorem~\ref{theorem1}, are used.
Then all temporal errors are bounded from above by
a moderate constant and decrease monotonically when
$\Delta t$ decreases.
This suggests an unconditionally stable behavior
of the schemes.
A further examination in this case shows that the
global temporal errors for the Do scheme can be
bounded from above by $C \Delta t$ and for the CS,
MCS, HV schemes by $C (\Delta t)^2$ (whenever
$\Delta t >0$) with constants $C$ depending
on the scheme.
\begin{figure}
\begin{center}
\begin{tabular}{c c}
         \includegraphics[width=0.5\textwidth]{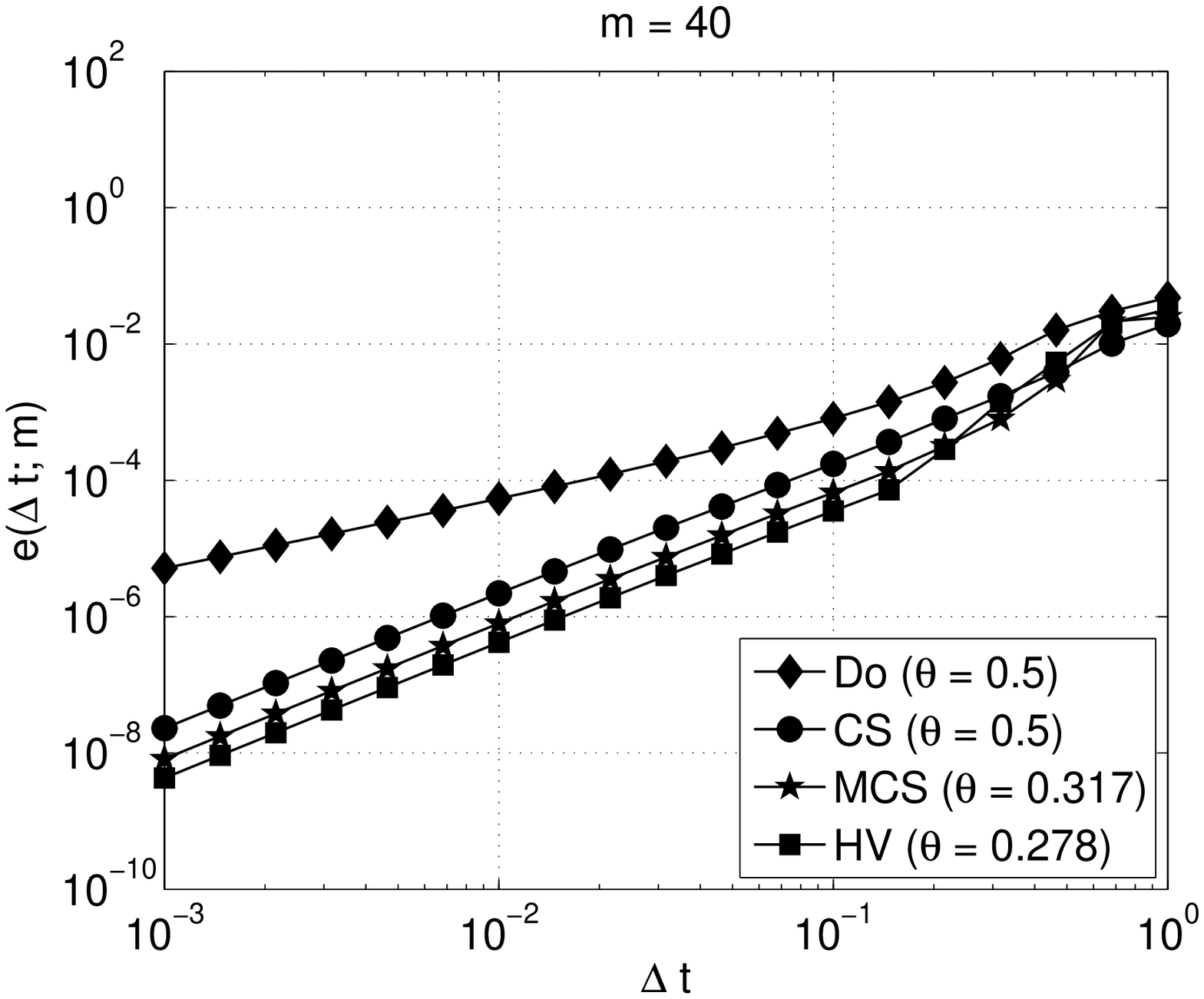}&
         \includegraphics[width=0.5\textwidth]{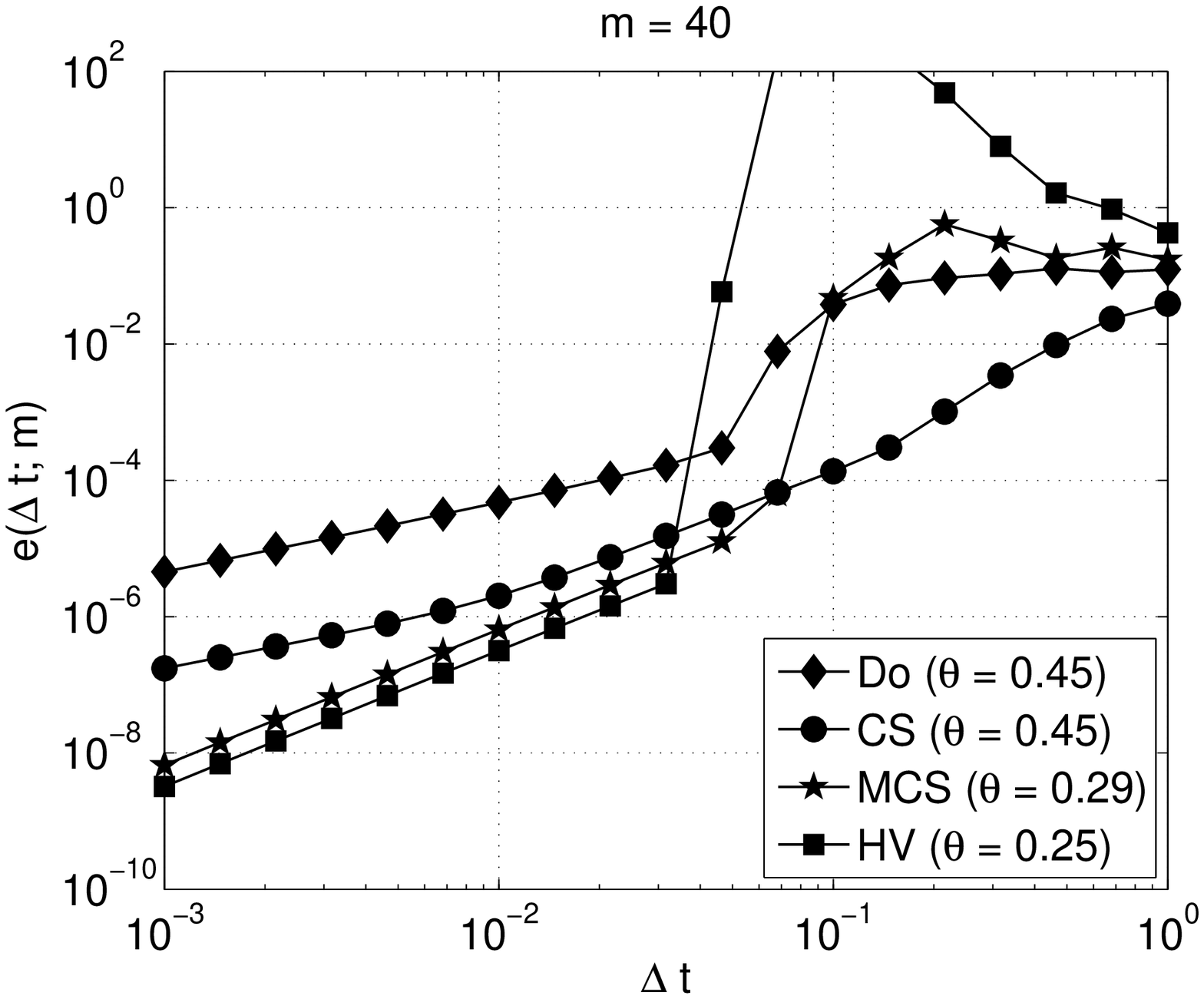}\\
         \includegraphics[width=0.5\textwidth]{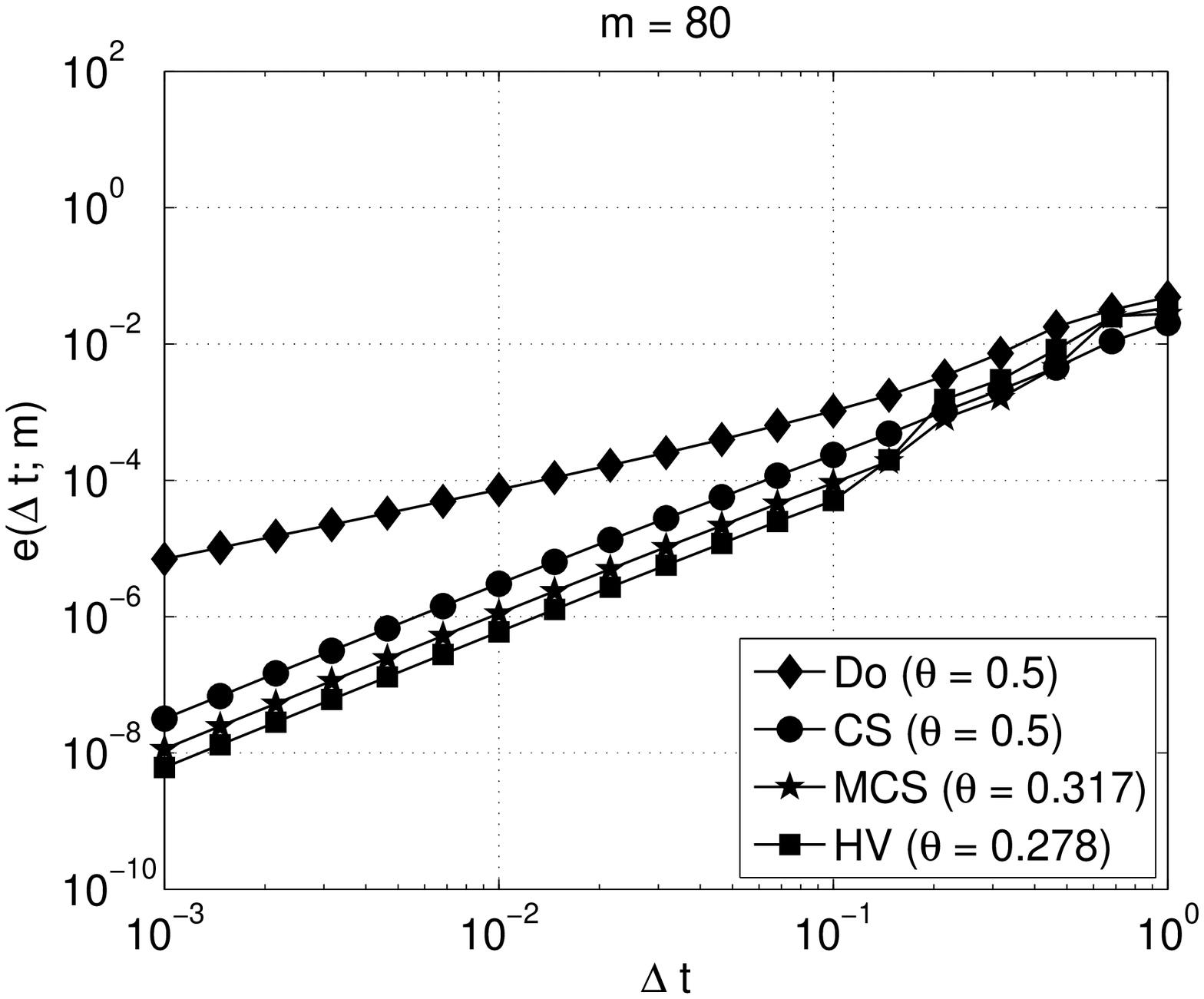}&
         \includegraphics[width=0.5\textwidth]{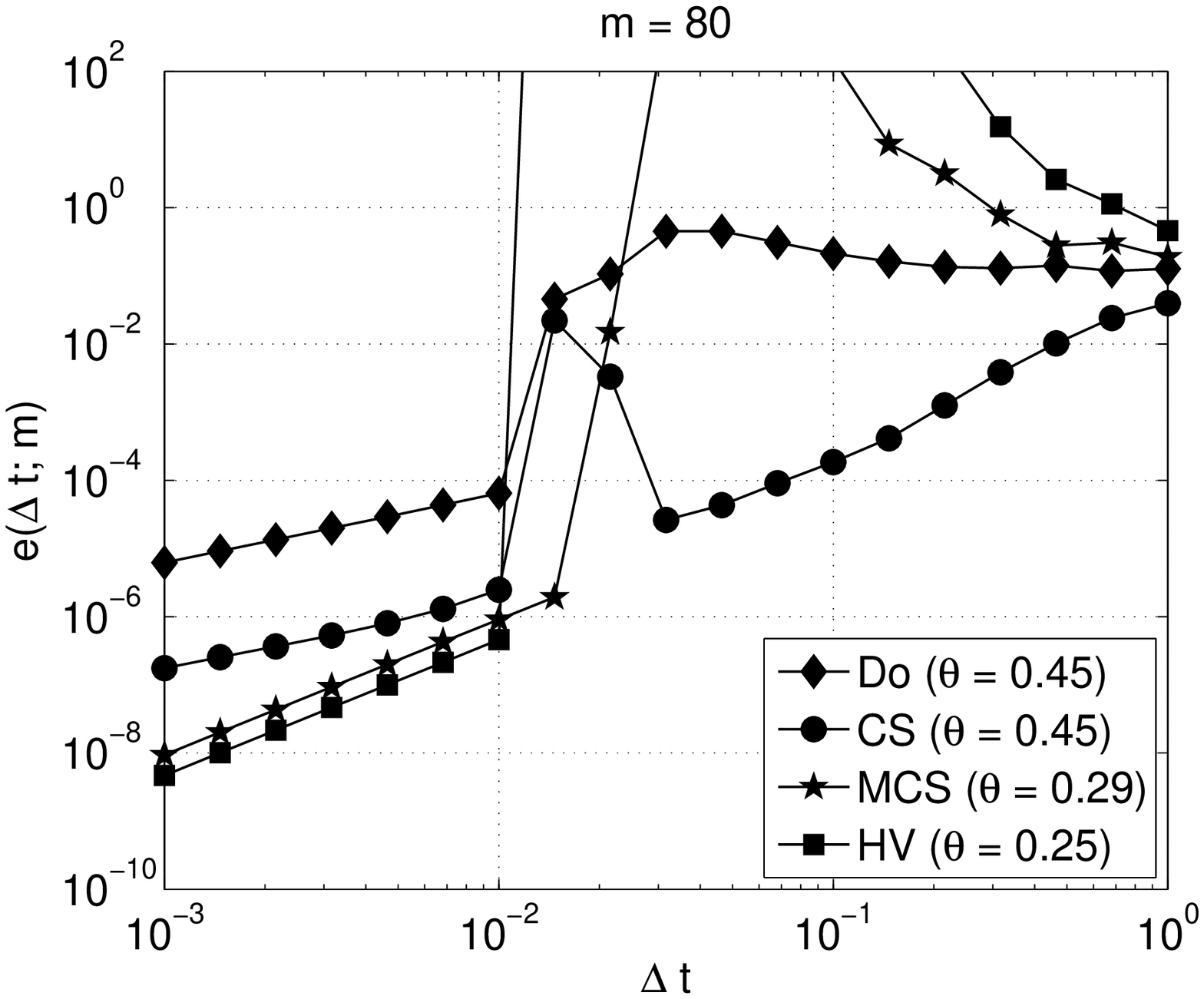}
\end{tabular}
\caption{Global temporal errors $e(\Delta t; m)$ versus $\Delta t$
for Do, CS, MCS, HV schemes when applied to a semidiscretized
2D problem (\ref{PDE}) with $\gamma=0.9$.
Top row: $m=40$.
Bottom row: $m=80$.
Left column: lower bound values $\theta$ given by
Theorem~\ref{theorem1}.
Right column: values $\theta$ that are about 90\% of these.
}
\label{2Derrors}
\end{center}
\end{figure}
\noindent
This clearly agrees with the respective orders of
consistency of the schemes.
Moreover, for each scheme the constant $C$ is only
weakly dependent on the number of spatial grid points,
determined by $m$, indicating that the error bounds
are valid in a stiff, hence favorable, sense.
We remark that similar conclusions were found for
larger values $t$.
Further we numerically verified that upon increasing
$\theta$ from its lower bound value, the error
constant $C$ increases, as mentioned in the discussion
following Theorem~\ref{theorem1}.\\

The second experiment deals with (\ref{PDE}) in three
spatial dimensions. We take initial function
\begin{equation*}
\label{ic}
u(x_1,x_2,x_3,0) = e^{-[\cos^2(\pi x_1)+\cos^2(\pi x_2)+\cos^2(\pi x_3)]}~~~
(0\le x_1, x_2, x_3 \le 1)
\end{equation*}
and diffusion matrix $D$ given by
\[
D = 0.025
\left(
  \begin{array}{ccc}
    1 & 2\gamma & \gamma\\
    2\gamma & 4 & 2\gamma\\
    \gamma & 2\gamma & 1
  \end{array}
\right),
\]
which is positive semidefinite and such that condition
(\ref{gamma}) holds whenever $\gamma\in [0,1]$.
Here we take $\gamma=0.75$.
Then the lower bounds on $\theta$ given by
Theorem~\ref{theorem1} for the Do, MCS, HV schemes are,
rounded to three decimal places, equal to 0.556, 0.385,
0.335, respectively.
For the CS scheme the lower bound is always
$\tfrac{1}{2}$ in three spatial dimensions,
independently of $\gamma$.
To the semidiscrete systems with
$\Delta x_1 = \Delta x_2 = \Delta x_3 = 1/m$ for $m=40,\, 80$
we have applied the Do, CS, MCS, HV schemes using their
exact lower bound values~$\theta$ as well as the values
0.5, 0.45, 0.35, 0.3, respectively, which are approximately
90\% of these.
Figure~\ref{3Derrors} displays the normalized global temporal
errors
\[
e(\Delta t; m) = m^{-3/2}\, \left| U(5)-U_{5N} \right|_2
\]
for a sequence of step sizes $\Delta t = 1/N$ with
$10^{-3} \le \Delta t \le 10^0$.
\begin{figure}
\begin{center}
\begin{tabular}{c c}
         \includegraphics[width=0.5\textwidth]{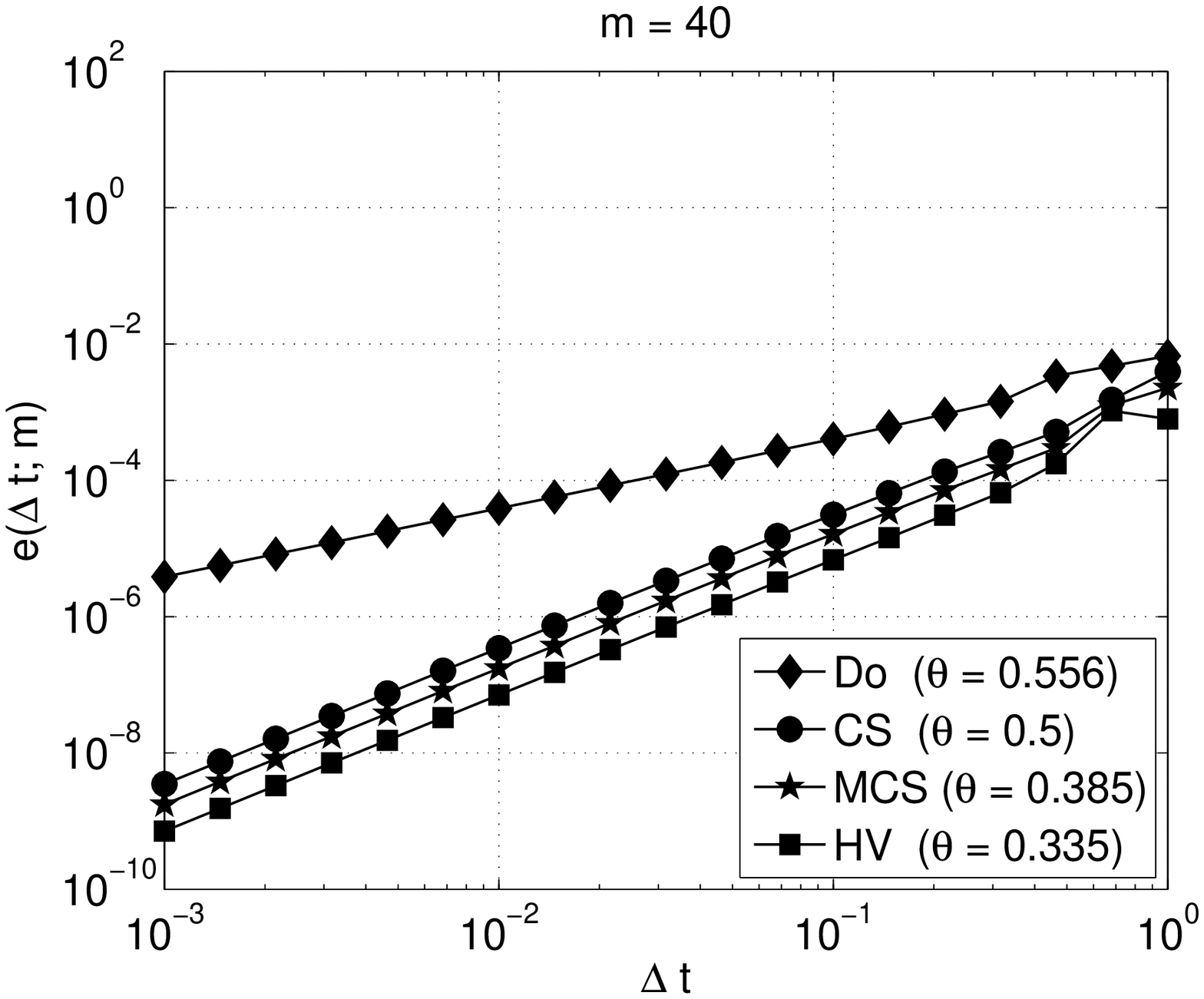}&
         \includegraphics[width=0.5\textwidth]{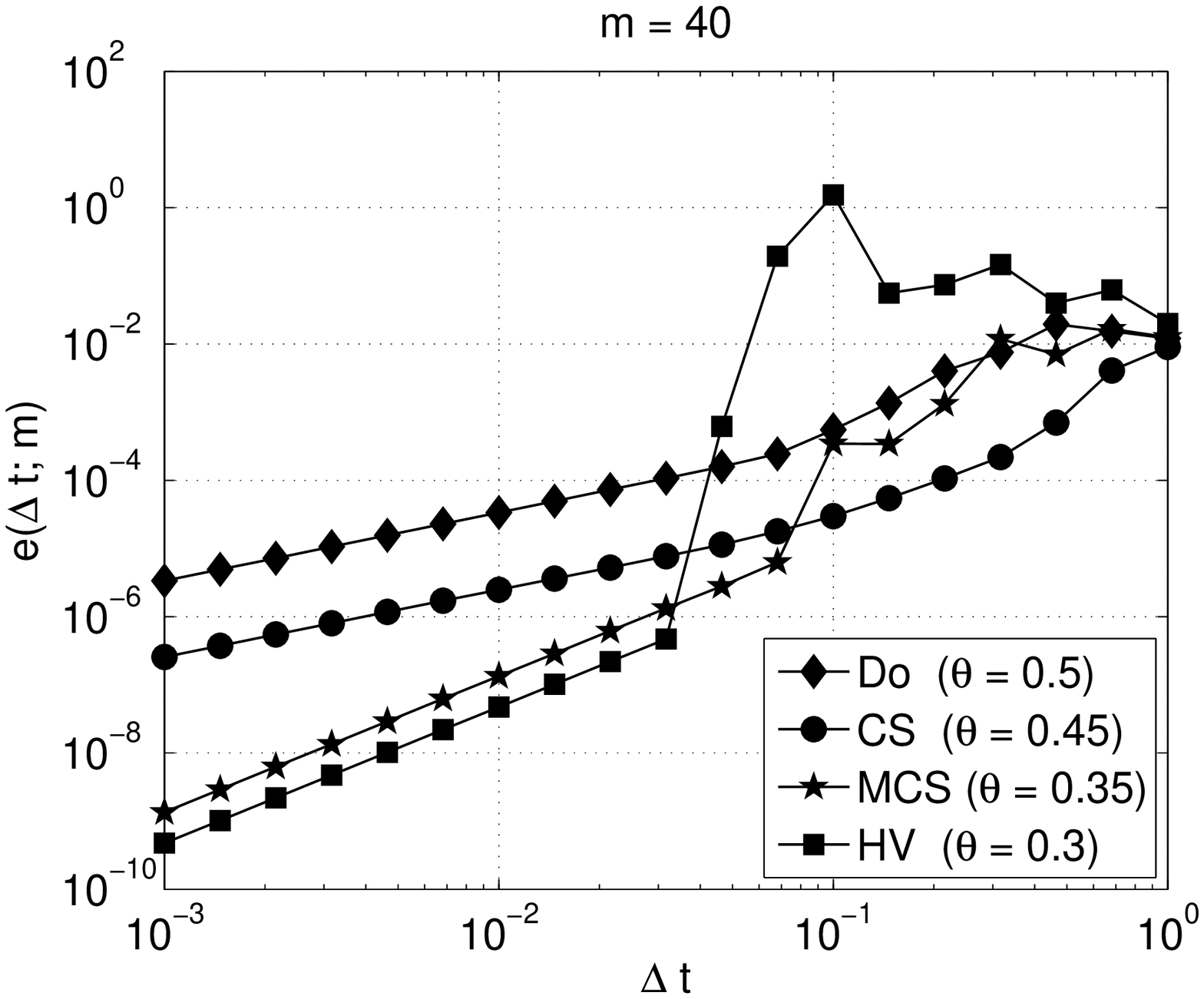}\\
         \includegraphics[width=0.5\textwidth]{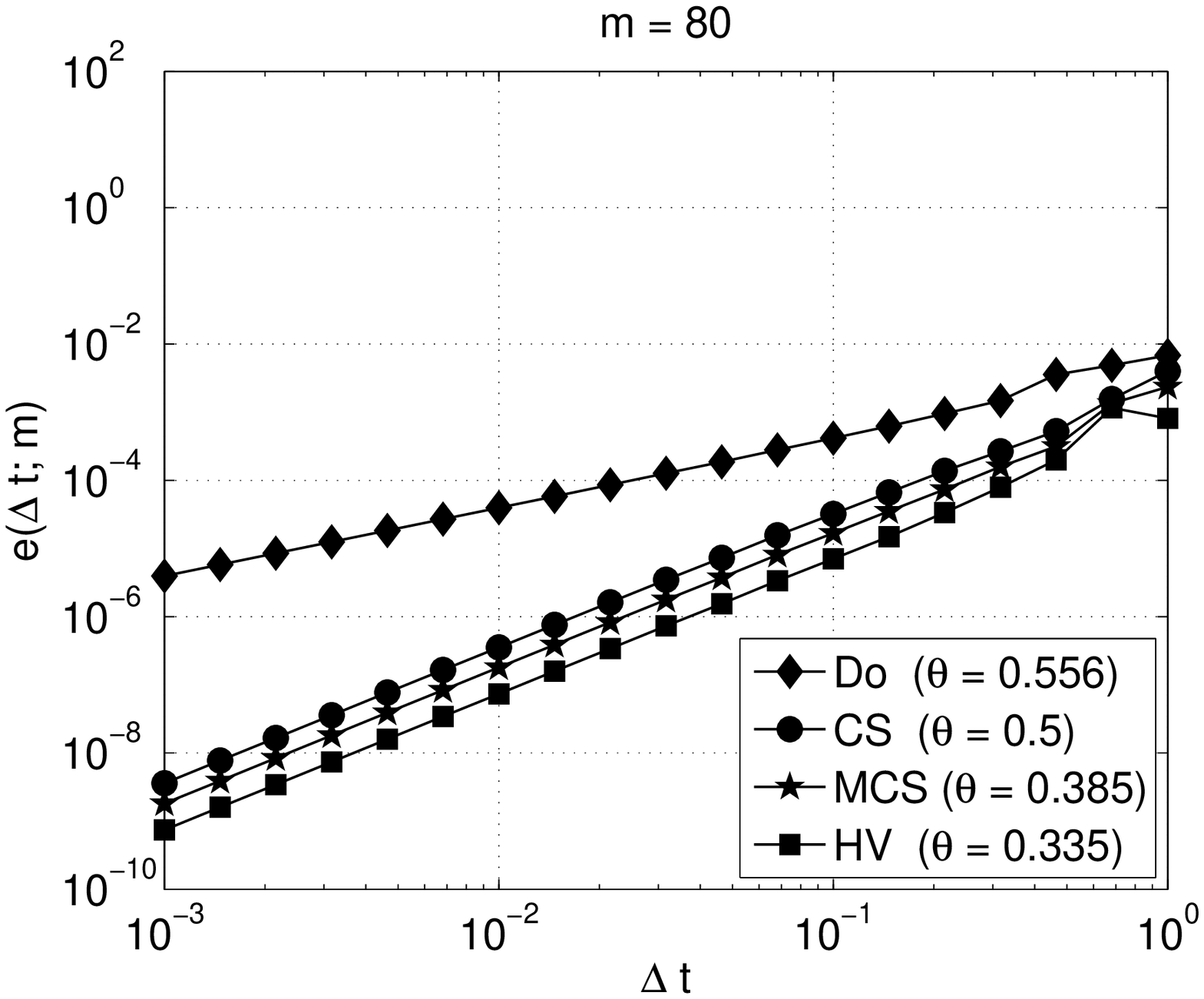}&
         \includegraphics[width=0.5\textwidth]{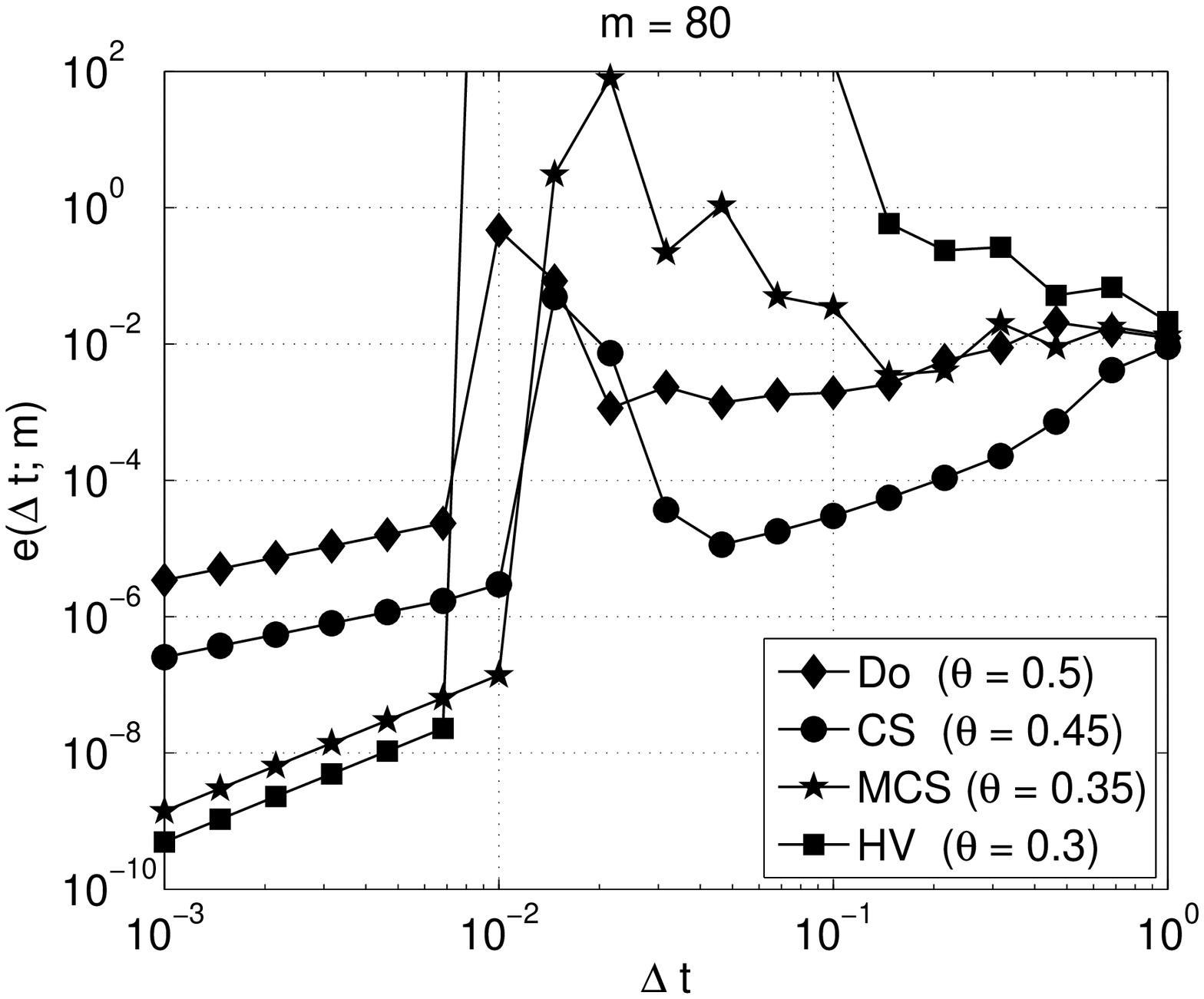}
\end{tabular}
\caption{Global temporal errors $e(\Delta t; m)$ versus
$\Delta t$
for Do, CS, MCS, HV schemes when applied to a semidiscretized
3D problem (\ref{PDE}) with $\gamma=0.75$.
Top row: $m=40$.
Bottom row: $m=80$.
Left column: lower bound values $\theta$ given by
Theorem~\ref{theorem1}.
Right column: values $\theta$ that are about 90\% of these.
}
\label{3Derrors}
\end{center}
\end{figure}
Exactly the same observations can be made as in the experiment
(above) for the two-dimensional case.
The large temporal errors for each ADI scheme when applied
with the smaller value $\theta$, as seen in the right
column of Figure~\ref{3Derrors}, correspond to instability
of the scheme.
When applied with their lower bound values $\theta$, given
by Theorem~\ref{theorem1}, the error behavior for all ADI
schemes is in agreement with unconditional stability of
the schemes.
Moreover, in this case a stiff order of convergence is
observed that is equal to one for the Do scheme and equal
to two for the CS, MCS and HV schemes.

\setcounter{equation}{0}
\setcounter{theorem}{0}
\section{Conclusions}\label{concl}
In this paper we analyzed stability in the von Neumann
sense of four well-known  ADI schemes - the Do, CS, MCS
and HV schemes - in the application to multidimensional
diffusion equations with mixed derivative terms.
Such equations are important, notably, to the field
of financial mathematics.
Necessary and sufficient conditions have been derived
on the parameter $\theta$ for unconditional stability
of each ADI scheme by taking into account the actual
size of the mixed derivative terms, measured
by the quantity $\gamma \in [0,1]$.
Our two main theorems generalize stability results
obtained by Craig \& Sneyd~\cite{CS88} and In 't Hout
\& Welfert~\cite{IHW09}.
Ample numerical experiments have been presented,
illustrating the main theorems and also providing
insight into the convergence behavior of all
schemes.

Among issues for future research,
it is of much interest to derive sufficient
conditions on $\theta$ for unconditional
stability of each ADI scheme in arbitrary
spatial dimensions $k\ge 4$ and arbitrary
$\gamma \in [0,1]$.
Also, it is of much interest to extend the
results obtained in this
paper to equations with advection terms.
This leads to general complex, instead of
real, eigenvalues, which forms a nontrivial
feature for the analysis,
cf.~e.g.~\cite{IHM11,IHW07}.

\end{document}